\documentclass[11pt,a4paper]{article}
\usepackage{amssymb} 
\usepackage{amsmath} 
\usepackage{amsthm} 
\usepackage{hyperref}
\usepackage{todonotes}
\usepackage{bbm}

\usepackage{enumitem}
\setlist[itemize]{itemsep=0pt, topsep=1.5pt}
\setlist[enumerate]{itemsep=0pt, topsep=1.5pt}

\newtheorem{theorem}{Theorem}
\numberwithin{theorem}{section}
\newtheorem{proposition}[theorem]{Proposition}
\newtheorem{lemma}[theorem]{Lemma}
\newtheorem{corollary}[theorem]{Corollary}
\newtheorem{conjecture}[theorem]{Conjecture}

\newtheorem{problem}[theorem]{Problem}
\newtheorem*{chernoff}{Chernoff Bounds}
\newtheorem*{glll}{A General Local Lemma}

\newcommand{\Exp}{\,\mathbb{E}}
\renewcommand{\Pr}{\,\mathbb{P}}
\newcommand{\eps}{\varepsilon}

\DeclareMathOperator{\Bin}{Bin}

\renewcommand{\labelenumi}{\theenumi}

\title{Bipartite induced density in triangle-free graphs}

\author{
Wouter Cames van Batenburg
\thanks{Department of Computer Science, Universit\'e Libre de Bruxelles, Belgium. 
Email: \protect\href{mailto:wcamesva@ulb.ac.be}{\protect\nolinkurl{wcamesva@ulb.ac.be}}. Supported by an ARC grant from the Wallonia-Brussels Federation of Belgium.}
\and
R\'emi de Joannis de Verclos
\thanks{Department of Mathematics, Radboud University Nijmegen, Netherlands. 
Email: \protect\href{mailto:r.deverclos@math.ru.nl}{\protect\nolinkurl{r.deverclos@math.ru.nl}}. Supported by a Vidi grant (639.032.614) of the Netherlands Organisation for Scientific Research (NWO).}
\and
Ross J. Kang
\thanks{Department of Mathematics, Radboud University Nijmegen, Netherlands. 
Email: \protect\href{mailto:ross.kang@gmail.com}{\protect\nolinkurl{ross.kang@gmail.com}}. Supported by a Vidi grant (639.032.614) of the Netherlands Organisation for Scientific Research (NWO).}
\and
Fran\c{c}ois Pirot
\thanks{Department of Mathematics, Radboud University Nijmegen, Netherlands
and LORIA, Universit\'e de Lorraine, Nancy, France.
Email: \protect\href{mailto:francois.pirot@loria.fr}{\protect\nolinkurl{francois.pirot@loria.fr}}.}
}

\begin{document}

\date{}

\maketitle

\begin{abstract}
We prove that any triangle-free graph on $n$ vertices with minimum degree at least $d$ contains a bipartite induced subgraph of minimum degree at least $d^2/(2n)$. This is sharp up to a logarithmic factor in $n$.
Relatedly, we show that the fractional chromatic number of any such triangle-free graph is at most the minimum of $n/d$ and $(2+o(1))\sqrt{n/\log n}$ as $n\to\infty$.
This is sharp up to constant factors.
Similarly, we show that the list chromatic number of any such triangle-free graph is at most $O(\min\{\sqrt{n},(n\log n)/d\})$ as $n\to\infty$.

Relatedly, we also make two conjectures. First, any triangle-free graph on $n$ vertices has fractional chromatic number at most $(\sqrt{2}+o(1))\sqrt{n/\log n}$ as $n\to\infty$. Second, any  triangle-free graph on $n$ vertices has list chromatic number at most $O(\sqrt{n/\log n})$ as $n\to\infty$.

\end{abstract}

\section{Introduction}\label{sec:intro}

Our starting point is a conjecture of the third author together with Esperet and Thomass\'e, which would be sharp up to the choice of constant if true.
This enticing conjecture was inspired by a list colouring problem.

\begin{conjecture}[Esperet, Kang, and Thomass\'e~\cite{EKT19}]\label{conj:trianglebip}
There is a constant $C>0$ such that any triangle-free graph with minimum degree at least $d$ contains a bipartite induced subgraph of minimum degree at least $C\log d$.
\end{conjecture}
\noindent
Although the conjecture is new, it might be difficult.
Conjecture~\ref{conj:trianglebip} aligns with central challenges in combinatorics, especially about stable sets in triangle-free graphs.
For example, iterating a result of Ajtai, Koml\'os and Szemer\'edi~\cite{AKS81} for triangle-free graphs of given average degree implies Conjecture~\ref{conj:trianglebip} for any triangle-free graph with $O(d)$ maximum degree\footnote{For the avoidance of any shadow of a doubt for the reader, fix some $0<\eps<1/2$ to be specified later, and consider a triangle-free graph on $n$ vertices with minimum degree $d$ and maximum degree $O(d)$.
By the result in~\cite{AKS81}, every induced subgraph on at least $\eps n$ vertices has a stable set of size $x=\lceil C(\eps n \log d)/d\rceil$ for some fixed $C>0$.
Thus we can repeatedly extract, so long as at least $\eps n$ vertices remain, disjoint stable sets of size $x$, ultimately yielding a partial proper colouring with at most $(1-\eps)n/x\le (1-\eps)d/(C\eps\log d)$ colour classes, all of size $x$.
The number of edges incident to the remaining vertices is $O(\eps n d)$, which can be made smaller than $n d/4$ for $\eps$ chosen small enough.
By the pigeonhole principle, two of the colour classes induce a bipartite subgraph of average degree at least $(nd/4) /(\binom{(1-\eps)n/x}{2}\cdot (2x)) \ge
(C\eps\log d)/(4(1-\eps)^2)$, which in turn contains a bipartite induced subgraph of minimum degree $(C\eps\log d)/(8(1-\eps)^2)$, as desired.}. 
The same result (cf.~\cite[Thm.~3.4]{EKT19}) confirms Conjecture~\ref{conj:trianglebip} for any triangle-free graph on $n$ vertices provided $d= \Omega(n^{2/3}\sqrt{\log n})$ as $n\to\infty$.

Our main result is a stronger, near optimal version of this last statement.

\begin{theorem}\label{thm:main}
There are constants $C_1,C_2>0$ such that, for $0\le d\le n/2$,
\begin{itemize}
\item
any triangle-free graph on $n \ge 2$ vertices with minimum degree at least $d$ contains a bipartite induced subgraph of minimum degree at least $\max\{C_1 d\sqrt{\log n/n},d^2/(2n)\}$; and
\item
provided $n/d$ is large enough, there is a triangle-free graph on between $n/2$ and $n$ vertices with minimum degree at least $d$ such that every bipartite induced subgraph has minimum degree at most $\lceil C_2 d^2/n\rceil\log n$.
\end{itemize}
\end{theorem}
\noindent
Thus we observe the following phase transition behaviour: for $d$ there is a critical exponent of $n$ (namely $1/2$) above which we can be assured of bipartite induced minimum degree polynomially large in $n$ and below which we cannot.
Theorem~\ref{thm:main} resolves Problem~4.1 in~\cite{EKT19} up to a logarithmic factor.
Our constructions for near optimality are blow-ups of an adaptation of Spencer's construction for lower bounds on $R(3,t)$~\cite{Spe77} (see Section~\ref{sec:constructions}).

Due to a captivating connection between bipartite induced density and fractional colouring~\cite{EKT19}, through which high bipartite induced density is guaranteed by small fractional chromatic number (see Theorem~\ref{thm:fractional}), Theorem~\ref{thm:main} is closely related to the following extremal result. 

\begin{theorem}\label{thm:main,fractional}
There are constants $C_1,C_2>0$ such that, for $0< d\le n/2$, 
\begin{itemize}
\item
any triangle-free graph on $n$ vertices with minimum degree at least $d$ has fractional chromatic number at most $\min\{C_1\sqrt{n/\log n},n/d\}$;
\item
provided $n/d$ is large enough, there is a triangle-free graph on between $n/2$ and $n$ vertices with minimum degree at least $d$ and fractional chromatic number at least $C_2\min\{\sqrt{n/\log n},n/d\}$.
\end{itemize}
\end{theorem}
\noindent
This bound is basic, but has not appeared in the literature as far as we know. There is equality in the $n/d$ bound when $d=n/2$, in which case we must have a complete bipartite graph with two equal-sized parts. 
As for the $\Theta(\sqrt{n/\log n})$ bound, it is an interesting problem to sharpen the asymptotic constants.
In Section~\ref{sec:redux}, we combine a recent colouring result of Molloy~\cite{Mol17+} with the proof idea in Theorem~\ref{thm:main,fractional} to show the following as a first step.

\begin{theorem}\label{thm:redux}
As $n\to\infty$, any triangle-free graph on $n$ vertices has fractional chromatic number at most $(2+o(1))\sqrt{n/\log n}$.
\end{theorem}
\noindent
Note that if one could improve the factor $(2+o(1))$ to $(\sqrt{2}+o(1))$ (which we formally state as a conjecture below), then it would match the best to date asymptotic upper bound for the Ramsey numbers $R(3,t)$ due to Shearer~\cite{She83}.
By the final outcome of the triangle-free process~\cite{BoKe13+,FGM20}, Theorem~\ref{thm:redux} is sharp up to a $(2\sqrt{2}+o(1))$ factor.
(The triangle-free process also gives sharpness up to a constant factor in Theorem~\ref{thm:main,fractional}.)

Theorem~\ref{thm:main} is not far from optimal, but the best constructions we know so far are almost regular. As noted above, we already know Conjecture~\ref{conj:trianglebip} in the almost regular case. This motivates the following bound which improves on Theorem~\ref{thm:main} if the graph is irregular. We prove this in Section~\ref{sec:abovesquareroot}.

\begin{theorem}\label{th:abovesquareroot}
Any triangle-free graph on $n\ge2$ vertices with $m\ge1$ edges and $w_3$ (directed) three-edge walks contains a bipartite induced subgraph of minimum degree at least
$w_3/(4 n m)$.
\end{theorem}

\noindent
Observe that $w_3 \ge 2m d^2$ if the graph has minimum degree at least $d$ (with equality for $d$-regular graphs), and so Theorem~\ref{th:abovesquareroot} directly implies the $d^2/(2n)$ bound.
Put another way, Theorem~\ref{th:abovesquareroot} replaces the squared minimum degree term in the bound of Theorem~\ref{thm:main} by the average over all edges of the product of the two endpoint degrees.

Returning in some sense to the original motivation for this research, our explorations related to Theorem~\ref{thm:redux} have inspired a list colouring analogue of Theorems~\ref{thm:main} and~\ref{thm:main,fractional}, building upon similar ideas.
\begin{theorem}\label{thm:main,list}
There is a constant $C>0$ such that any triangle-free graph on $n$ vertices with minimum degree at least $d$ has list chromatic number at most $C\min\{\sqrt{n},(n\log n)/d\}$.
\end{theorem}
\noindent
We are unsure of how close this bound is to optimality, but in Section~\ref{sec:list} we offer speculation of the correct asymptotic order, in particular for the bound purely in terms of $n$.

Theorems~\ref{thm:main,fractional} and~\ref{thm:main,list} inspire the following natural question.
\begin{problem}
Given a function $d =d(n)$, letting $\chi(n,d)$ denote the largest chromatic number of a triangle-free graph on $n$ vertices with minimum degree at least $d$, asymptotically what is $\chi(n,d)/\min\{\sqrt{n/\log n},n/d\}$ as $n\to\infty$?
\end{problem}

\noindent
In a sense, substantial effort has already been devoted to this problem when $d$ is linear in $n$, in relation to a problem of Erd\H{o}s and Simonovits~\cite{ErSi73}, see e.g.~\cite{Tho02,BrTh11+}. In particular, for $\eps>0$ fixed, the answer is $\omega(1)$ if $d < (1/3-\eps)n$ and $d=\Omega(n)$, while it is $O(1)$ if $d > (1/3+\eps)n$ or $d = O(\sqrt{n\log n})$.
By Theorem~\ref{thm:main,list} (or Proposition~\ref{prop:logn} below), the ratio is always $O(\log n)$.

\paragraph{Outline of the paper.} 
This paper is organised as follows. In Section~\ref{sec:lower} we prove Theorem~\ref{thm:main,fractional}, which in turn implies the lower bound of Theorem~\ref{thm:main} (via Theorem~\ref{thm:fractional}). In Section~\ref{sec:constructions}, we use the Local Lemma to construct a random graph that certifies the near optimality of Theorem~\ref{thm:main}. In the remaining sections we explore generalisations of Theorems~\ref{thm:main} and \ref{thm:main,fractional}. In Section~\ref{sec:redux}, we prove Theorem~\ref{thm:redux} and we also provide bounds in terms of the number of edges. Section~\ref{sec:abovesquareroot} is devoted to the proof of Theorem~\ref{th:abovesquareroot} and related results involving the fractional chromatic number of the cube of triangle-free graphs. We discuss and prove Theorem~\ref{thm:main,list} in Section~\ref{sec:list}.
Finally, in Section~\ref{sec:largercliques} we make some concluding remarks concerning the exclusion of an arbitrary subgraph. In particular, we show how our results generalise to excluding any given cycle as a subgraph.

\subsection{Probabilistic preliminaries}\label{sub:probabilistic}

We use some specific forms of the Chernoff bound~\cite[(2.9) and~(2.11)]{JLR00}
and the Lov\'asz Local Lemma~\cite[Thm.~1.3]{Spe77}.
\begin{chernoff}
If $0\le\eps \le 1$, then \[\Pr(|\Bin(n,p) - np| \ge \eps np) < 2\exp(-\eps^2np/3).\]
If $x\ge 7np$, then $\Pr(\Bin(n,p) \ge x) < \exp(-x)$.
\end{chernoff}

\begin{glll}
Consider a set ${\cal E}=\{A_1,\dots,A_n\}$ of (bad) events such that
each $A_i$ is mutually independent of ${\cal E}-({\cal D}_i\cup A_i)$, for some ${\cal D}_i \subseteq {\cal E}$.
If we have reals $y_1,\dots,y_n >0$ such that for each $i$
\[
y_i\Pr(A_i)<1 \quad \text{and} \quad \log y_i > \sum_{A_j\in {\cal D}_i} y_j\Pr(A_j),
\]
then the probability that none of the events in $\cal E$ occur is positive.
\end{glll}

\section{Fractional colouring}\label{sec:lower}

Given a graph $G=(V,E)$, we say that a probability distribution $\mathcal{S}$ over the stable sets of $G$ satisfies property ${\rm Q}^*_r$ if $\Pr(v \in \mathbf{S})\ge r$ for every $v\in V$ and $\mathbf{S}$ taken randomly according to $\mathcal{S}$.
Recall that the \emph{fractional chromatic number $\chi_f(G)$} of $G$ is defined as the smallest $k$ such that there is a probability distribution over the stable sets of $G$ satisfying property ${\rm Q}^*_{1/k}$.

\begin{proof}[Proof of Theorem~\ref{thm:main,fractional}]
Let $G=(V,E)$ be a triangle-free graph on $n$ vertices with minimum degree at least $d$.
Since the fractional chromatic number is at most the chromatic number, the first term of the upper bound was already observed by Erd\H{o}s and Hajnal~\cite{ErHa85} as a consequence of the aforementioned result of Ajtai, Koml\'os and Szemer\'edi~\cite{AKS81} (see Section~\ref{sec:redux}).  
For the second term of the upper bound, choose $\mathbf{S}$ from full neighbourhood sets uniformly over all $n$ such sets. Since $G$ is triangle-free,  $\mathbf{S}$ is a stable set.
For all $v\in V$,
\[
\Pr(v\in \mathbf{S}) = \frac{\deg(v)}{n} \ge \frac{d}{n}.
\]
We have shown then that this distribution has property ${\rm Q}^*_{d/n}$, as required.

For sharpness, fix $\eps>0$ and let $j$ be the smaller of $n$ and the least value for which $2(2+\eps/4)^{1/2}\sqrt{j/\log j}\ge n/d$, and consider the final output of the triangle-free process on $j$ vertices. This is a random triangle-free graph that was shown, independently, by Bohman and Keevash~\cite{BoKe13+} and by Fiz Pontiveros, Griffiths and Morris~\cite{FGM20}, to have minimum degree $(\sqrt{2^{-1}}+o(1))\sqrt{j\log j}$ and stability number at most $(\sqrt{2}+o(1))\sqrt{j\log j}$ with high probability as $j\to\infty$. For large enough $j$ (which we can guarantee if $n/d$ is large enough), we may fix a triangle-free graph $\hat{G}$ on $j$ vertices that has minimum degree at least $(2+\eps/4)^{-1/2}\sqrt{j\log j}$ and stability number at most $(2+\eps/4)^{1/2}\sqrt{j\log j}$. Form a new graph $G$ from $\hat{G}$ by replacing each vertex by a stable set of size $\lfloor n/j \rfloor$, and adding a complete bipartite graph between every pair of stable sets that corresponds to an edge in $\hat{G}$.
Observe that $G$ is a triangle-free graph on between $n/2$ and $n$ vertices with minimum degree at least $(2+\eps/4)^{-1/2}\sqrt{j\log j}\lfloor n/j \rfloor \ge d$. Moreover, $G$ has  stability number at most $(2+\eps/4)^{1/2}\sqrt{j\log j}\lfloor n/j \rfloor \le 2(2+\eps/4)d$ and so has fractional chromatic number at least $(n/2)/(2(2+\eps/4)d)=n/((8+\eps)d)$, as desired.

(Note that we essentially lost a factor $2$ twice due to rounding, which is only an issue when $d=\Theta(\sqrt{n\log n})$. Thus when $d=\omega(\sqrt{n\log n})$ the $n/d$ upper bound is in fact correct up to a factor of $(2+o(1))$ as $n\to\infty$, and so in this case we can take the choice $C_2=1/2+o(1)$ in the theorem statement.)

For the other lower bound, recall that the triangle-free process on $n$ vertices yields, with high probability, a random graph with stability number at most $(\sqrt{2}+o(1))\sqrt{n\log n}$. Thus the fractional chromatic number of that graph is at least $(2+o(1))^{-1/2} \sqrt{n/\log n}$.
\end{proof}

\section{Bipartite induced density}\label{sec:constructions}

Here is a link between fractional colouring and bipartite induced density, and hence between Theorems~\ref{thm:main} and~\ref{thm:main,fractional}.

\begin{theorem}[Esperet, Kang, and Thomass\'e~\cite{EKT19}]\label{thm:fractional}
Any graph with fractional chromatic number at most $k$ and average degree $d$ has a bipartite induced subgraph of average degree at least $d/k$.
\end{theorem}

\noindent
There is always a subgraph whose minimum degree is at least half the graph's average degree. Thus 
%
the lower bound in Theorem~\ref{thm:main} follows from the upper bound in Theorem~\ref{thm:main,fractional}.

Our next task, and the main task of this section, is to prove near sharpness in Theorem~\ref{thm:main,fractional}.

However, instead of a dense bipartite \emph{induced subgraph}, one might be satisfied with a dense bipartite \emph{subgraph} where we only require that (at least) one of the two parts induces a stable set.
Given $G=(V,E)$, we call an induced subgraph $G'=(V',E')$ of $G$ {\em semi-bipartite} if it admits a partition $V'=V_1\cup V_2$ such that $V_1$ is a stable set of $G$, and we define the degree of a vertex of $G'$ with respect to the semi-bipartition as its degree in the bipartite subgraph $G[V_1,V_2]$ between $V_1$ and $V_2$ (and so we ignore any edges in $V_2$).
A version of Conjecture~\ref{conj:trianglebip} where `bipartite' is replaced by `semi-bipartite' is known~\cite{EKT19}.

In what follows we give near optimal constructions for Theorem~\ref{thm:main} not only for bipartite induced density, but also for semi-bipartite induced density.
The following result, an adaptation of work of Spencer~\cite{Spe77}, is central. It might also be possible to adapt an earlier construction due to Erd\H{o}s~\cite{Erd61}, but it would produce a construction comparable to Theorem~\ref{thm:main}.
Although we chose not to pursue it, we suspect that the outcome at the end of the triangle-free process 
 is significantly better, and optimal up to a constant factor.
For this reason, we did not optimise either of the constants below.

\begin{theorem}\label{thm:Spe77}
There exist constants $\delta,\gamma>0$ such that for every large enough $n$ there is a triangle-free graph on $n$ vertices with minimum degree at least $\delta\sqrt{n}$ that contains no semi-bipartite induced subgraph of minimum degree at least $\gamma\log n$.
\end{theorem}

\noindent
Before proving this, let us see how it implies the second part of Theorem~\ref{thm:main}.

\begin{proof}[Proof of sharpness in Theorem~\ref{thm:main}]
Let $j$ be the smaller of $n$ and the least value for which $2\sqrt{j}\ge \delta n/d$.
Provided $j$ is large enough, we may by Theorem~\ref{thm:Spe77} fix a triangle-free graph $\hat{G}$ that has minimum degree at least $\delta\sqrt{j}$ that contains no semi-bipartite induced subgraph of minimum degree at least $\gamma\log j$. Consider a new graph $G$ formed from $\hat{G}$ by replacing each vertex by a stable set of size $\lfloor n/j \rfloor$, and adding a complete bipartite graph between every pair of stable sets that corresponds to an edge in $\hat{G}$.
Note that $G$ is a triangle-free graph on $j$ between $n/2$ and $n$ vertices with minimum degree at least $\delta\sqrt{j}\lfloor n/j \rfloor \ge d$. Moreover, the largest minimum degree of a semi-bipartite induced subgraph in $G$ is smaller than $\gamma\lfloor n/j \rfloor\log j \le  \gamma \lceil\delta^2d^2/(16n)\rceil\log n$.
\end{proof}

For Theorem~\ref{thm:Spe77} we will need the convenient observation that, if we do not mind constant factors, it suffices to consider only semi-bipartite induced subgraphs with both parts of equal size.

\begin{proposition}[\cite{EKT19}]\label{prop:reduction}
Suppose $A,B\subseteq G$ are disjoint with $|A|\ge |B|$ and satisfy that the average degree of $G[A,B]$ is $d$.
Then there exists $A'\subseteq A$ with $|A'|=|B|$ such that the average degree of $G[A',B]$ is at least $d/2$.
\end{proposition}
%

\begin{proof}[Proof of Theorem~\ref{thm:Spe77}]
For a sufficiently large positive integer $n$,
let $p = c_1/\sqrt{n}$ and $t=c_2 \sqrt{n}\log n$ for some fixed $c_1,c_2>0$.
Consider the binomial random graph $G(n,p)$.
Fix $0<\alpha<1$ and $\beta >0$.
The constants $c_1,c_2,\alpha,\beta$ will be specified more precisely later in the proof.

With a view to applying the General Local Lemma (as stated in Subsection~\ref{sub:probabilistic}), let us define four types of (bad) events in $G(n,p)$.
\begin{enumerate}
\renewcommand\labelenumi{\Alph{enumi}}
\item
For a set of three vertices, it induces a triangle.
\item
For a set of $t$ vertices, it induces a stable set.
\item
For a single vertex, it has degree at most $(1-\alpha) n p$.
\item[D$_i$]
For two disjoint sets of $i$ vertices, the bipartite subgraph induced by the cut between the two sets has average degree at least $\beta \log n$.
\end{enumerate}
Note that by Proposition~\ref{prop:reduction}, the fact that average degree is at least minimum degree, and the choice of $p$, we obtain the desired graph if there is an element of the probability space $G(n,p)$ for which no event of Types~A,~B,~C, and~D$_i$, $\beta \log n \le i \le t$, occur.

By an abuse of notation, let us write $P(*)$ for the probability of an event of Type~$*$. To justify this abuse, observe that all events of Type~$*$ have the same probability of occurrence.
We have that $P(\text{A})=p^3$
and $P(\text{B})=(1-p)^{\binom{t}{2}} < \exp\left(-p\binom{t}{2}\right)$.
Note that $P(\text{C}) = \Pr(\Bin(n,p) \le (1-\alpha)np) \le 2\exp(-\alpha^2np/3)$ by a Chernoff Bound (as stated in Subsection~\ref{sub:probabilistic}).
Since $ip \le tp \le c_1c_2\log n$, $P(\text{D}_i) = \Pr(\Bin(i^2,p) \ge \beta i \log n) \le \exp(-\beta i \log n)$ for all $i\le t$ by a Chernoff Bound with a choice of $\beta$ satisfying
\begin{align}\label{eqn:1}
\beta\ge 7 c_1c_2.
\end{align}

For Types~$*$ and~$*'$, 
let us write $N(*,*')$ for the least such number for which each Type~$*$ event is mutually independent of all but $N(*,*')$ events of Type~$*'$.
We have that $N(\text{A},\text{A}) =3(n-3)<3n$,
$N(\text{B},\text{A}) = \binom{t}{2}(n-t)+\binom{t}{3} < t^2n/2$, 
$N(\text{C},\text{A}) = \binom{n-1}{2} < n^2/2$, and
$N(\text{D}_i,\text{A}) < i^2 n$.
More crudely, we can simply bound $N(*,*')$ by the number of Type~$*'$ events.
This yields
$N(\text{A},\text{B})$, $N(\text{B},\text{B})$, $N(\text{C},\text{B})$, and $N(\text{D}_i,\text{B})$ are all at most $\binom{n}{t}<(en/t)^t=\exp(t\log(en/t))$;
$N(\text{A},\text{C})$, $N(\text{B},\text{C})$, $N(\text{C},\text{C})$, and $N(\text{D}_i,\text{C})$ are all at most $n$; and
$N(\text{A},\text{D}_j)$, $N(\text{B},\text{D}_j)$, $N(\text{C},\text{D}_j)$, and $N(\text{D}_i,\text{D}_j)$ are all at most $\binom{n}{2j}<(en/2j)^{2j}=\exp(2j\log(en/2j))$,

By the General Local Lemma, we only need to find, for $\beta \log n \le i \le t$, positive reals $Y(\text{A})$, $Y(\text{B})$, $Y(\text{C})$, and $Y(\text{D}_i)$ such that 
the following inequalities hold:
\begin{align*}
1 >
& Y(\text{A})P(\text{A}), Y(\text{B})P(\text{B}),Y(\text{C})P(\text{C}),Y(\text{D}_i)P(\text{D}_i); \\
\log Y(\text{A}) > 
& Y(\text{A})P(\text{A})N(\text{A},\text{A}) + Y(\text{B})P(\text{B})N(\text{A},\text{B})+\\
& +Y(\text{C})P(\text{C})N(\text{A},\text{C})+\sum_j Y(\text{D}_j)P(\text{D}_j)N(\text{A},\text{D}_j);\\
\log Y(\text{B}) > 
& Y(\text{A})P(\text{A})N(\text{B},\text{A}) + Y(\text{B})P(\text{B})N(\text{B},\text{B})+\\
& +Y(\text{C})P(\text{C})N(\text{B},\text{C})+\sum_j Y(\text{D}_j)P(\text{D}_j)N(\text{B},\text{D}_j);\\
\log Y(\text{C}) > 
& Y(\text{A})P(\text{A})N(\text{C},\text{A}) + Y(\text{B})P(\text{B})N(\text{C},\text{B})+\\
& +Y(\text{C})P(\text{C})N(\text{C},\text{C})+\sum_j Y(\text{D}_j)P(\text{D}_j)N(\text{C},\text{D}_j);\quad\text{and}\\
\log Y(\text{D}_i) > 
& Y(\text{A})P(\text{A})N(\text{D}_i,\text{A}) + Y(\text{B})P(\text{B})N(\text{D}_i,\text{B})+\\
& +Y(\text{C})P(\text{C})N(\text{D}_i,\text{C})+\sum_j Y(\text{D}_j)P(\text{D}_j)N(\text{D}_i,\text{D}_j).
\end{align*}
The estimates that we derived earlier imply that it instead suffices to find positive reals $Y(\text{A})$, $Y(\text{B})$, $Y(\text{C})$, and $Y(\text{D}_i)$ for which the following hold:
\begin{align*}
&1 >
 Y(\text{A})p^3, \frac{Y(\text{B})}{\exp\left(p\binom{t}{2}\right)},\frac{2Y(\text{C})}{\exp(\alpha^2np/3)},\frac{Y(\text{D}_i)}{\exp(\beta i \log n)}; \\
&Z<\log Y(\text{A}) - Y(\text{A})p^3\cdot3n;\\
&Z<\log Y(\text{B}) - Y(\text{A})p^3t^2n/2;\\
&Z<\log Y(\text{C}) - Y(\text{A})p^3n^2/2;\text{ and}\\
&Z<\log Y(\text{D}_i) - Y(\text{A})p^3i^2/n;
\end{align*}
where
\begin{align*}
Z=Y(\text{B})\exp\left(-p\binom{t}{2}+t\log\frac{en}{t}\right)+Y(\text{C})\cdot 2\exp(-\alpha^2np/3)\cdot n+\\
 +\sum_j Y(\text{D}_j)\exp(-\beta j \log n+2j\log(en/2j)).
\end{align*}
Let us choose $Y(\text{A}) = 1+\eps$, $Y(\text{B}) = \exp(c_3\sqrt{n}\log^2 n)$, $Y(\text{C})=\exp(c_4\sqrt{n})$, and $Y(\text{D}_i)=\exp(c_5i\log n)$ for some fixed $\eps,c_3,c_4,c_5>0$.
Note by the choices of $p$ and $t$ that the first row of inequalities are easily satisfied.

For the remaining rows, we first consider the asymptotic behaviour of the three constituents of $Z$ as $n\to\infty$. For the first two, we have
\begin{align*}
Y(\text{B})\exp\left(-p\binom{t}{2}+t\log\frac{en}{t}\right)
& = \exp((c_3-c_1c_2^2/2+c_2+o(1))\sqrt{n}\log^2 n)
\\
\text{and } Y(\text{C})\cdot 2\exp(-\alpha^2np/3)\cdot n
& = \exp((c_4-\alpha^2c_1/3 +o(1))\sqrt{n}).
\end{align*}
For the third, note since $\beta \log n \le j \le t$ that
\begin{align*}
&\sum_j Y(\text{D}_j)\exp(-\beta j \log n+2j\log(en/2j))\\
& \le \sum_j \exp((c_5-\beta+2+o(1))j\log n)
 \le \exp((c_5-\beta+2+o(1))\beta\log^2 n).
\end{align*}
We may therefore conclude that $Z$ is superpolynomially small in $n$ provided
\begin{align}
c_3-c_1c_2^2/2+c_2 & < 0,\label{eqn:2}\\
c_4-\alpha^2c_1/3 & < 0,\text{ and} \label{eqn:3}\\
c_5-\beta+2 & < 0. \label{eqn:4}
\end{align}

Of the remaining terms in the inequalities required for the application of the General Local Lemma, the critical ones can be seen to be polylogarithmic or greater in magnitude (as $n\to\infty$):
\begin{align*}
&\log Y(\text{B}) - Y(\text{A})p^3t^2n/2 = (c_3-(1+\eps)c_1^3c_2^2/2+o(1))\sqrt{n}\log^2 n;\\
&\log Y(\text{C}) - Y(\text{A})p^3n^2/2 = (c_4-(1+\eps)c_1^3/2+o(1))\sqrt{n};\text{ and}\\
&\log Y(\text{D}_i) - Y(\text{A})p^3i^2/n \ge (c_5-(1+\eps)c_1^3c_2)i\log n
\end{align*}
(where we used $i\le c_2\sqrt{n}\log n$ in the last line).

We therefore also want that
\begin{align}
c_3-(1+\eps)c_1^3c_2^2/2 & > 0;\label{eqn:5}\\
c_4-(1+\eps)c_1^3/2 & > 0;\text{ and}\label{eqn:6}\\
c_5-(1+\eps)c_1^3c_2 & > 0.\label{eqn:7}
\end{align}

It remains only to show that there is some choice of $c_1,\dots,c_5,\alpha,\beta,\eps$ so that~\eqref{eqn:1}--\eqref{eqn:7} are fulfilled.
Note that, whatever the other choices, the inequalities~\eqref{eqn:1},~\eqref{eqn:4}, and~\eqref{eqn:7} are satisfied with a sufficiently large choice of $\beta$ or of $c_5$.
By pairing inequalities~\eqref{eqn:2} and~\eqref{eqn:5} as well as~\eqref{eqn:3} and~\eqref{eqn:6}, we need
\begin{align*}
c_2(c_1c_2/2-1) > c_3 & > (1+\eps)c_1^3c_2^2/2\text{ and}\\
\alpha^2c_1/3 > c_4 & > (1+\eps)c_1^3/2.
\end{align*}
We have that 
\begin{align*}
c_2(c_1c_2/2-1) > (1+\eps)c_1^3c_2^2/2 &\iff c_2 > 2/(c_1-(1+\eps)c_1^3)\text{ and}\\
\alpha^2c_1/3 > (1+\eps)c_1^3/2 &\iff \alpha > \sqrt{3(1+\eps)/2}\cdot c_1.
\end{align*}
Let us then fix $c_1 = 1/\sqrt{3}$. Therefore with a small enough choice of $\eps >0$ it is possible to choose, say, $c_2=21/4$ and $\alpha=3/4$ ($<1$ particular) and then take, say, $c_3=0.51$ and $c_4=0.97$. 
This completes the proof.
\end{proof}

\section{Fractional colouring revisited}\label{sec:redux}

In this section, we make a first step towards optimising the asymptotic constant for the first term in Theorem~\ref{thm:main,fractional}.
It turns out that this is related to two problems of Erd\H{o}s and Hajnal~\cite{ErHa85}, concerning the asymptotic order of the chromatic number of a triangle-free graph with a given number of vertices or edges. 
In terms of edges, the correct order upper bound was first shown by Poljak and Tuza~\cite{PoTu94}.
Matching lower bounds to settle both problems were established as byproduct to the determination of the asymptotic order of the Ramsey numbers $R(3,t)$ by Kim~\cite{Kim95}, cf.~\cite{GiTh00,Nil00}.

For completeness we reiterate more precisely the observation of Erd\H{o}s and Hajnal~\cite{ErHa85} mentioned in the proof of Theorem~\ref{thm:main,fractional} (see also~\cite[pp.~124--5]{JeTo95} and~\cite{Kim95}).
An application of Shearer's lower bound on the stability number~\cite{She83} in a greedy colouring procedure bounds the chromatic number by at most $(4+o(1))$ times optimal (as certified by the triangle-free process~\cite{BoKe13+,FGM20}).

\begin{lemma}[cf.~Jensen and Toft~\cite{JeTo95}]\label{lem:JeTo95}
Let $\cal G$ be a class of graphs that is closed under vertex-deletion. Suppose that for some $x_0 \ge 2$ there is a continuous, non-decreasing function $f_{\cal G}: [x_0,\infty) \to {\mathbb R}^+$ such that every $G'\in \cal G$ on $x \ge x_0$ vertices has a stable set of at least $f_{\cal G}(x)$ vertices.
Then every $G\in \cal G$ on $n \ge x_0$ vertices has chromatic number at most
\begin{align*}
x_0+ \int_{x_0}^n \frac{dx}{f_{\cal G}(x)}.
\end{align*}
\end{lemma}

\begin{corollary}\label{cor:She83}
As $n\to\infty$, any triangle-free graph on $n$ vertices has chromatic number at most $(2\sqrt{2}+o(1))\sqrt{n/\log n}$.
\end{corollary}

\begin{proof}
Shearer~\cite{She83} showed that, for any $\eps > 0$, there exists $x_0 \ge 2$ such that the function $f_{\cal G}(x) = (1/\sqrt{2}-\eps)\sqrt{x \log x}$ satisfies the hypothesis of Lemma~\ref{lem:JeTo95} for $\cal G$ being the class of triangle-free graphs.
Lemma~\ref{lem:JeTo95} yields the desired outcome after an exercise in analysis to show that
\[
\lim_{n\to \infty} \frac{\int_{x_0}^n dx/\sqrt{x\log x}}{\sqrt{n/\log n}} = 2. \qedhere
\]
\end{proof}

Is the factor $(2+o(1))$ contribution from the above limit truly necessary? We were unable to address this issue, but Theorem~\ref{thm:redux} shows that it is possible to reduce the bound by a factor $(\sqrt{2}+o(1))$ if we only wish to bound the fractional chromatic number. 
Note that definitive progress on whether it is possible to improve by {\em strictly} more than a factor $(2+o(1))$ in Corollary~\ref{cor:She83}, for fractional or not, either positively or negatively, would likely constitute a major breakthrough in combinatorics.
A factor $(2+o(1))$ improvement is indeed plausible, especially for fractional.

\begin{conjecture}\label{conj:redux}
As $n\to\infty$, any triangle-free graph on $n$ vertices has fractional chromatic number at most $(\sqrt{2}+o(1))\sqrt{n/\log n}$.
\end{conjecture}

\noindent
Relatedly, in the spirit of~\cite{ErHa85}, we also conjecture the following. Like for the previous conjecture, this if true would be an analogue of Shearer's bound.
\begin{conjecture}\label{conj:redux,edges}
As $m\to\infty$, any triangle-free graph with $m$ edges has fractional chromatic number at most $(2^{4/3}+o(1))m^{1/3}/(\log m)^{2/3}$.
\end{conjecture}


Here it is essential to mention a recent achievement of Molloy~\cite{Mol17+}, a simpler derivation of and improvement upon Johansson's theorem~\cite{Joh96a}.

\begin{theorem}[Molloy~\cite{Mol17+}]\label{thm:Molloy}
As $\Delta\to\infty$, any triangle-free graph of maximum degree at most $\Delta$ has list chromatic number at most $(1+o(1))\Delta/\log \Delta$.
\end{theorem}

\noindent
Since this may be considered a stronger form of Shearer's bound, one might wonder if it alone is enough to verify Conjecture~\ref{conj:redux}. This does not seem to be the case. 
We remark however that Theorems~\ref{thm:main,fractional} and~\ref{thm:Molloy} together immediately yield Conjectures~\ref{conj:redux} and~\ref{conj:redux,edges} for {\em regular} triangle-free graphs.
(Let $G$ be a $D$-regular triangle-free graph. If $D \ge \sqrt{2^{-1}n\log n}$, then it follows from Theorem~\ref{thm:main,fractional}; otherwise, it follows from Theorem~\ref{thm:Molloy}.)
Moreover, as we will shortly see, an iterated application of Theorem~\ref{thm:Molloy} combined with the simple idea in the proof of Theorem~\ref{thm:main,fractional} yields Theorem~\ref{thm:redux}.

Recall that we may equivalently define the fractional chromatic number of a graph as the smallest $k$ such that there is an assignment of measurable subsets of the interval $[0,k]$ (or rather of any subset of $\mathbb R$ of measure $k$) to the vertices such that each vertex is assigned a subset of measure $1$ and subsets assigned to adjacent vertices are disjoint.

\begin{proof}[Proof of Theorem~\ref{thm:redux}]
Fix $\eps>0$.
Without loss of generality, assume $\eps<1/2$. Let $G=(V,E)$ be a triangle-free graph on $n$ vertices and let $D \le n$ be some positive integer to be specified later in the proof. We first associate $n$ disjoint intervals of measure $1/D$ to each of the full neighbourhood sets (each of which is a stable set), and assign each such interval to its neighbourhood's vertices. By independently, arbitrarily de-assigning some (parts) of these intervals, we may assume each vertex of degree at least $D$ has an assignment of measure exactly $1$. 
On the other hand, the subgraph induced by vertices of measure less than $1$ has maximum degree less than $D$.
More precisely, let $V_i$ be the set of vertices of degree exactly $i$ in $G$, for $i < D$:
this initial partial fractional colouring gives each vertex of $V_i$ an assignment of measure exactly $i/D<1$.
We have essentially shown how it suffices to restrict our attention in the remainder of the proof to $G$ having maximum degree $D$, by an incorporation of the same idea used in the proof of Theorem~\ref{thm:main,fractional}.

For each $D^{1/(1+\eps/5)} \le i < D$, let us write $G_i$ for the subgraph of $G$ induced by $\cup_{j=0}^i V_j$. Since $G_i$ is a triangle-free graph of maximum degree at most $i$, it follows from Theorem~\ref{thm:Molloy} that $G_i$ admits a proper colouring $c_i$ of its vertices with at most $(1+\eps/5)i/\log i$ colours, provided $i$ is large enough. (Since $i \ge D^{1/(1+\eps/5)}$, $i$ is arbitrarily large if $D$ is.) For each colour class $\mathcal{C}$ of $c_i$, we choose an interval of measure $1/D$ (that is disjoint from all previously used intervals), and assign it to each vertex of $\mathcal{C}$.

This extends the initial partial fractional colouring to nearly all of $G$. If $D^{1/(1+\eps/5)} \le i < D$, then each vertex of $V_i$ has been assigned $D-i$ additional intervals of measure $1/D$, resulting in an assignment of measure $1$.
Note that, for $D$ large enough, the total measure of the subsets we have thus used is
\begin{align*}
 \frac{n}{D} &+ \frac{1}{D}\sum_{i=\lceil D^{1/(1+\eps/5)}\rceil}^D\frac{(1+\eps/5)i}{\log i} \\
 & \le  \frac{n}{D} + \frac{(1+\eps/5)^2}{D\log D}\sum_{i=0}^D i
  = \frac{n}{D} + \frac{(1+\eps/5)^2(D+1)}{2\log D}\\
 & \le \frac{n}{D} + \frac{(1+\eps/2)D}{2\log D}.
\end{align*}

We have extended the initial partial fractional colouring so that every vertex of $G$ has measure $1$ apart from those vertices of degree less than $D^{1/(1+\eps/5)}$. Since the above bound on the total measure used is strictly more than $D^{1/(1+\eps/5)}$ if $D$ is large enough, we can greedily extend the partial fractional colouring to all remaining vertices without any additional measure.

It remains to specify $D$ so that we use at most $(\sqrt{2}+\eps)\sqrt{n/\log n}$ measure in total. Provided $n$ is large enough, the choice $D= \lfloor\sqrt{n\log n}\rfloor$ suffices. (Note that under this choice $D$ is arbitrarily large if $n$ is.)
\end{proof}




To conclude the section, we comment that a straightforward substitution of Theorem~\ref{thm:redux} or Corollary~\ref{cor:She83} together with Theorem~\ref{thm:Molloy} into the proof by Gimbel and Thomassen~\cite{GiTh00} (the proofs in~\cite{Nil00,PoTu94} being slightly less efficient) yields the following bounds. The constants are roughly $2.5$ and $3$ times larger than the constant in Conjecture~\ref{conj:redux,edges}.

\begin{proposition}\label{prop:edges}
As $m\to\infty$, any triangle-free graph with $m$ edges has fractional chromatic number at most $(3^{5/3}+o(1))m^{1/3}/(\log m)^{2/3}$ and chromatic number at most $(3^{5/3}2^{1/3}+o(1))m^{1/3}/(\log m)^{2/3}$.
\end{proposition}

\noindent
The approach for Theorem~\ref{thm:redux} could possibly be adapted to more directly improve upon Proposition~\ref{prop:edges}, but we have not yet managed to do so.

\section{Bounds involving cubes}\label{sec:abovesquareroot}

In this section, we prove Theorem~\ref{th:abovesquareroot}. We also make some additional observations that link our results with the fractional distance-$3$ chromatic number.

Let us first remark that, given the adjacency matrix $A$ of a graph $G$, the total number of directed three-edge walks in $G$ is the sum of all entries in the matrix $A^3$.
The proof of Theorem~\ref{th:abovesquareroot} combines ideas from the proofs of Theorems~\ref{thm:main,fractional} and of~\ref{thm:fractional}, without needing to bound the fractional chromatic number.

\begin{proof}[Proof of Theorem~\ref{th:abovesquareroot}]
Let $G=(V,E)$ be a triangle-free graph with $|V|=n$ and $|E|=m$ and suppose $G$ has $w_3$ directed three-edge walks. We write $q = w_3/(2nm)$ and
note that
\[
q = \frac{\sum_{x\in V} \sum_{v \in N(x)} \sum_{w\in N(v)} \deg(w)}{n \sum_{x \in V} \deg(x)}.
\]
Let $s_1,s_2$ be two vertices chosen uniformly at random and let $\mathbf{S}_1:=N(s_1)$ and $\mathbf{S}_2:=N(s_2)$ denote their neighbourhoods, which are stable sets by triangle-freeness.
Note that
\(
\Exp(|\mathbf{S}_1|)=\Exp(|\mathbf{S}_2|)=\frac{1}{n} \sum_{x\in V} \deg(x),
\)
so also
\(
\frac{1}{2}\Exp(|\mathbf{S}_1|+|\mathbf{S}_2)|)=\frac{1}{n} \sum_{x\in V} \deg(x).
\)

The number of edges in the subgraph induced by $\mathbf{S}_1\cup\mathbf{S}_2$ satisfies
\begin{align*}
&\Exp\left(|E(G[\mathbf{S}_1\cup\mathbf{S}_2])|  \right) \\
&= \sum_{S_1,S_2\subseteq [n]} \Pr\left(  (\mathbf{S}_1=S_1) \cap (\mathbf{S}_2=S_2) \right) \cdot |E(G[S_1\cup S_2])|\\
&= \sum_{S_2} \Pr\left(\mathbf{S}_2=S_2 \right) \sum_{S_1} \Pr\left(\mathbf{S}_1=S_1  \right) \cdot |E(G[S_1\cup S_2])|\\
&= \sum_{S_2} \Pr\left(\mathbf{S}_2=S_2 \right) \cdot \Exp\left( |E(G[\mathbf{S}_1 \cup S_2])|\right)
= \frac{1}{n} \sum_{x \in V} \Exp \left( |E(G[\mathbf{S}_1 \cup N(x)])| \right) \\
&= \frac{1}{n} \sum_{x \in V} \sum_{v\in N(x)} \Exp \left( |\mathbf{S}_1 \cap N(v)| \right) 
= \frac{1}{n} \sum_{x \in V} \sum_{v \in N(x)} \sum_{w \in N(v)} \frac{\deg(w)}{n} \\
&= \frac{\sum_{x\in V} \sum_{v \in N(x)} \sum_{w\in N(v)} \deg(w)}{n \sum_{x \in V} \deg(x)} \cdot \frac{\Exp(|\mathbf{S}_1|+|\mathbf{S}_2|)}{2} 
= \frac{q}{2} \Exp(|\mathbf{S}_1|+|\mathbf{S}_2|).
\end{align*}
By linearity of expectation, 
\[
\Exp\left(|E(G[\mathbf{S}_1\cup\mathbf{S}_2])|  - \frac{q}{2} (|\mathbf{S}_1|+ |\mathbf{S}_2|) \right) \geq 0.
\]

It follows that there are two stable sets $S_1$ and $S_2$ of $G$ with at least $\tfrac{q}{2} \left( |S_1|+|S_2|\right)$ edges in the subgraph induced by $S_1 \cup S_2$. Discarding the vertices of $S_1 \cap S_2$ (if any exist) yields a bipartite induced subgraph of average degree at least $q$. Therefore $G$ contains a bipartite induced subgraph of minimum degree at least $q/2$, as desired.
\end{proof}

Next we indicate a mild improvement upon our bounds in terms of fractional distance-$3$ colouring. Given a graph $G$, the cube $G^3$ of $G$ is the simple graph formed from $G$ by including all edges between vertices that are connected by a path in $G$ of length at most $3$.
The fractional distance-$3$ chromatic number of $G$ is the fractional chromatic number $\chi_f(G^3)$ of $G^3$.
Observe that, if $G$ is triangle-free and $S$ is a stable set of $G^3$, then the union $\cup_{v\in S} N_G(v)$ of neighbourhood sets taken over $S$ is a stable set in $G$.
In the proofs of Theorems~\ref{thm:main,fractional} and~\ref{th:abovesquareroot}, if we sample stable sets by taking such neighbourhood unions according to the distribution given by $\chi_f(G^3)$ rather than uniformly taking a neighbourhood set, then we obtain the following.

\begin{theorem}\label{thm:distance3}
We have that
\begin{itemize}
\item the upper bound in Theorem~\ref{thm:main,fractional} holds with $\chi^3_f/d$ instead of $n/d$; and
\item Theorem~\ref{th:abovesquareroot} holds with $w_3/(4 \chi^3_f m)$ or $d^2/(2\chi^3_f)$ instead of $w_3/(4 n m)$;
\end{itemize}
where in each case $G$ denotes the corresponding triangle-free graph and $\chi^3_f=\chi_f(G^3)$ denotes its fractional distance-$3$ chromatic number.
\end{theorem}

Curiously, as the triangle-free process is sharp in Theorem~\ref{thm:main,fractional}, we obtain the following. (Perhaps this same result with distance-$2$ also holds.)

\begin{corollary}
With high probability, the final output of the triangle-free process has $\Omega(n)$ fractional distance-$3$ chromatic number as $n\to\infty$.
\end{corollary}

\section{List colouring}
\label{sec:list}

In Section~\ref{sec:redux}, we pursued sharper but fractional versions of the original problems of Erd\H{o}s and Hajnal~\cite{ErHa85}. In another direction, the natural list colouring versions are open to the best of our knowledge.

\begin{conjecture}\label{conj:list}
There are constants $C_1,C_2>0$ such that any triangle-free graph on $n$ vertices with $m$ edges has list chromatic number at most $C\sqrt{n/\log n}$ and at most $C_2m^{1/3}/(\log m)^{2/3}$.
\end{conjecture}

\noindent
Note that by a result of Alon~\cite{Alo92}, the two terms in Conjecture~\ref{conj:list} are correct up to $\log n$ and $\log m$ factors, respectively.

Moreover, using a similar approach as for Theorem~\ref{thm:redux}, here we give progress towards one of the statements in Conjecture~\ref{conj:list}, as embodied by Theorem~\ref{thm:main,list}. Theorem~\ref{thm:main,list} is a combination of the bounds in Corollary~\ref{cor:list} and Theorem~\ref{thm:list} stated below.

We need a modest refinement of the aforementioned result of Alon.

\begin{proposition}\label{prop:logn}
Any graph on $n\ge2$ vertices with fractional chromatic number $\chi_f$ has list chromatic number at most $\lceil \chi_f \log n\rceil$.
\end{proposition}

\begin{proof}
This adapts a standard argument that we include for completeness.
Let $G=(V,E)$ be a graph on $n\ge 2$ vertices with fractional chromatic number $\chi_f$.
Let $k=\lceil \chi_f \log n\rceil$ and $L$ be a $k$-list-assignment of $G$. We write $L(V)=\bigcup_{u\in V}L(u)$ for the set of colours listed by $L$.
By a standard equivalent definition of fractional colouring, there is a proper $(a,b)$-colouring $c$ of $G$ with $a/b=\chi_f$; that is, there is an assignment of subsets of $\{1,\dots,a\}$ of size $b$ to the vertices of $G$ such that adjacent vertices are assigned disjoint subsets.

For each $x\in L(V)$ let $\mathbf{y}(x)$ be uniformly drawn from $\{1,\dots,a\}$, 
and for each $u\in V$ let $\mathbf{L}_c(u)=L(u)\cap\bigcup_{i\in c(u)}\mathbf{y}^{-1}(i)$;
that is, $\mathbf{L}_c(u)$ consists of those elements of $L(u)$ whose random choice from $\{1,\dots,a\}$ is an element of the $b$-element set $c(u)$.
Since necessarily $\mathbf{L}_c(u)\cap \mathbf{L}_c(v)=\emptyset$ if $uv\in E$,
there is a proper $L$-colouring of $G$ if it holds that $\mathbf{L}_c(u) \ne \emptyset$ for all $u\in V$. But this follows from a union bound and the probabilistic method since for all $u$
\[
\Pr(\mathbf{L}_c(u)=\emptyset) = (1-b/a)^k< e^{-k/\chi_f}\le1/n.\qedhere
\]
\end{proof}


Together with Theorem~\ref{thm:main,fractional}, this has the following direct consequence.

\begin{corollary}\label{cor:list}
Any triangle-free graph on $n$ vertices with minimum degree at least $d$ has list chromatic number at most $\lceil(n\log n)/d\rceil$.
\end{corollary}

\begin{theorem}\label{thm:list}
As $n\to\infty$, any triangle-free graph on $n$ vertices has list chromatic number at most $(2\sqrt{2}+o(1))\sqrt{n}$.
\end{theorem}

\begin{proof}
Fix $\eps>0$.
Without loss of generality, assume $\eps<1/2$. Let $G=(V,E)$ be a triangle-free graph on $n$ vertices and choose $D = (1+o(1)) \sqrt{n/2}\cdot \log n$ such that $(n\log n)/D=D/\log D=(1+o(1))\sqrt{2n}$.
Just as in the proof of Theorem~\ref{thm:redux}, we associate $n$ disjoint intervals of measure $1/D$ to each of the full neighbourhood sets (each of which is a stable set), and assign each such interval to its neighbourhood's vertices. Each vertex of degree more than $D$ has an assignment of measure greater than $1$. 

Writing $G_{\le D}$ for the subgraph of $G$ induced by the vertices of degree at most $D$, and $G_{>D} = G- V(G_{\le D})$, we have just observed that $G_{>D}$ has fractional chromatic number less than $n/D$.
By Proposition~\ref{prop:logn} and Theorem~\ref{thm:main,fractional}, it thus has list chromatic number at most $\lceil(n\log n)/D\rceil$.
On the other hand, $G_{\le D}$ has maximum degree $D$ and thus by Theorem~\ref{thm:Molloy} has list chromatic number at most $(1+\eps/4)D/\log D$ provided $D$ is large enough.

Let
\[k=\left\lceil (1+\eps)^2\left(\frac{n\log n}{D}+\frac{D}{\log D}\right) \right\rceil
\]
and $L$ be a $k$-list-assignment of $G$. We write $L(V)=\bigcup_{u\in V}L(u)$ for the set of colours listed by $L$.
Next 
for each $x\in L(V)$ independently at random include $x$ in the list $\mathbf{L}_{\le D}$ with probability $1/2$, and otherwise include $x$ in the list $\mathbf{L}_{>D}$.
By a Chernoff Bound, for each $v\in V$,
\begin{align*}
\Pr\left( |L(v) \cap \mathbf{L}_{\le D}| < \left(1+\frac{\eps}{4}\right)\frac{D}{\log D}\right) 
& \le \Pr\left( |L(v) \cap \mathbf{L}_{\le D}| \le (1-\eps)\frac{k}{2}\right) \\
&\le 2\exp(-\eps^2k/6)= o(1/n)
\end{align*}
as $n\to\infty$, using the choice of $D$ and the assumption that $\eps<1/2$.
Similarly,
\[
\Pr\left( |L(v) \cap \mathbf{L}_{>D}| < \left\lceil \frac{n\log n}{D} \right\rceil\right)=o(1/n).
\]
A union bound and the probabilistic method guarantees for sufficiently large $n$ the existence of list-assignments $L_{\le D}$ and $L_{>D}$ of $G$ such that the following properties hold for all $v\in V$:
\begin{itemize}
\item $L(v) = L_{\le D}(v) \cup L_{>D}(v)$ and $L_{\le D}(v) \cap L_{>D}(v) = \emptyset$, and
\item $|L_{\le D}(v)|\ge(1+\eps/4)D/\log D$ and  $|L_{>D}(v)| \ge \lceil(n\log n)/D\rceil$.
\end{itemize}
By the observations we made earlier, there is a proper $L_{\le D}$-colouring of $G_{\le D}$ and there is a proper $L_{>D}$-colouring of $G_{>D}$. Since these list-assignments are disjoint and are sub-list-assignments of $L$, their combination constitutes a proper $L$-colouring of $G$.
Since $k=((1+\eps)^2+o(1))2\sqrt{2n}$, an arbitrarily small choice of $\eps$ gives the result.
\end{proof}

\section{Concluding remarks}
\label{sec:largercliques}

Although we were preoccupied with triangle-free graphs, one could naturally investigate graphs not containing $H$ as a subgraph for any fixed graph $H$. The following is in essence a more general form of Problem~4.1 in~\cite{EKT19}.

\begin{problem}\label{prob:H}
Given a graph $H$, is there $c_H \in (0,1)$ such that, as $n\to\infty$,
\begin{itemize}
\item if $c> c_H$, then any $H$-free graph on $n$ vertices with minimum degree $n^c$ has $n^{\Omega(1)}$ bipartite induced minimum degree; and
\item if $c< c_H$, then there is an $H$-free graph on $n$ vertices with minimum degree $n^c$ and $O(\log n)$ bipartite induced minimum degree?
\end{itemize}
\end{problem}
\noindent
We have shown that $c_H=1/2$ if $H$ is a triangle.

Problem~\ref{prob:H} is particularly enticing when $H$ is the complete graph $K_r$ on $r\ge4$ vertices. It was noted in~\cite{EKT19} that the work of Ajtai, Koml\'os and Szemer\'edi~\cite{AKS81} implies $c_{K_r} \le 1-1/r$ (if $c_{K_r}$ exists).
It is possible to adapt Theorem~\ref{thm:Spe77} and~\cite{Spe77} to show that $c_{K_r} \ge 1-(r-2)\left/\left(\binom{r}{2}-1\right)\right.$ (if $c_{K_r}$ exists).
It is conceivable that $c_{K_r} = 1-1/(r-1)$.
A motivation for this is that, even though {\em prima facie} there is no extremely close connection between bipartite induced density and large stable sets, we are tempted to speculate that, denoting the $H$ versus $K_t$ Ramsey number by $R(H,K_t)$,
\[
c_H = 1-\lim_{t\to\infty} \frac{\log t}{\log R(H,K_t)} \text{ (if $c_H$ and the limit exist)}.
\]
The righthand side is conjectured to be $1-1/(r-1)$ when $H$ is $K_r$.

Motivated by Problem~\ref{prob:H}, we observe the following partial extensions of the bounds in Theorems~\ref{thm:main} and~\ref{thm:main,fractional}.

\begin{proposition}\label{le:tripartite}
Fix an integer $r\ge 3$. For $r-2\le d \le n/2$, any $K_{1,1,r-2}$-free graph on $n$ vertices with minimum degree at least $d$ has fractional chromatic number at most $\binom{n}{r-2}/\binom{d}{r-2}$, and thus contains a bipartite induced subgraph of minimum degree at least $\frac{d}{2} \binom{d}{r-2} / \binom{n}{r-2}$.
\end{proposition}
\begin{proof}
Let $G=(V,E)$ be a graph on $n$ vertices with minimum degree at least $d$ that contains no copy of $K_{1,1,r-2}$. 
We note that $K_{1,1,r-2}$-freeness implies that the joint neighbourhood of every vertex subset of size $r-2$ is a stable set.
Choose $\mathbf{S}$ from the joint neighbourhood sets of $(r-2)$-vertex subsets uniformly over all $\binom{n}{r-2}$ such sets.
Then it holds for all $v\in V$ that
\[
\binom{n}{r-2} \cdot \Pr\left( v \in \mathbf{S}  \right) =|\{ T \subseteq N(v) \mid |T|=r-2 \}|\ge \binom{d}{r-2}.
\]
We have shown then that this distribution has property ${\rm Q}^*_{\binom{d}{r-2}\left/\binom{n}{r-2}\right.}$.

Note that the second part follows from Theorem~\ref{thm:fractional}.
\end{proof}

Given a graph $H=(V_H,E_H)$ and a positive integer $x$, let us define 
\[
\chi_f(H,x)= \max_{v \in V_H}\max \left\{\chi_f(J) \mid J \text{ is an $(H-v)$-free graph on $x$ vertices} \right\}.
\]
\begin{proposition}\label{prop:usinglocalcolouring}
Given a graph $H$, any $H$-free graph on $n$ vertices with minimum degree at least $d$ and maximum degree at most $\Delta$ has fractional chromatic number
at most $\chi_f(H,\Delta)\cdot n/d$.
\end{proposition}
\begin{proof}
Let $G=(V,E)$ be a graph on $n$ vertices with minimum degree at least $d$ and maximum degree at most $\Delta$ that contains no copy of $H$. 
 For any $v\in V$, there is  by definition of $\chi_f(H,x)$ a distribution $\mathcal{S}_v$ over the stable sets of $N(v)$ such that any given $w\in N(v) $ is in $\mathbf{S}_v$ with probability at least $(\chi_f(H,\deg(v)))^{-1} \geq (\chi_f(H,\Delta))^{-1}$ for a random $\mathbf{S}_v$ chosen according to $\mathcal{S}_v$. Let $\mathbf{S}$ be $\mathbf{S}_v$, where $v$ is uniformly chosen from $V$. Then any given $u\in V$ is in $\mathbf{S}$ with probability at least $\deg(u)/(n \cdot \chi_f(H,\Delta))$, so this distribution has property ${\rm Q}^*_{d /(n \cdot \chi_f(H,\Delta))}$.
\end{proof}

Let $P_r$ ($C_r$) denote a path (cycle, respectively) on $r\ge 2$ vertices. Every $P_r$-free graph is $(r-2)$-degenerate and therefore $(r-1)$-colourable. Thus $\chi_f(C_{r+1},x)\leq r-1$ and we have the following corollary.
\begin{corollary}
Fix an integer $r\ge 3$. Any $C_r$-free graph on $n$ vertices with minimum degree $d$ has fractional chromatic number at most $(r-2)n/d$, and thus contains a bipartite induced subgraph of minimum degree at least $d^2/(2(r-2)n)$.
\end{corollary}

\subsubsection*{Notes added} Shortly after posting our manuscript to a public preprint repository, we learned that Matthew Kwan, Benny Sudakov and Tuan Tran independently obtained a finer version of Theorem~\ref{thm:main} with different methods. 
In particular, they proved a bipartite induced minimum degree of order $\log d$ if $d=n^{\Omega(1)}$ and $d\le \sqrt{n}$, and found a better construction in the case $d\ge n^{2/3}$. 
Later with Shoham Letzter~\cite{KLST18+} they moreover proved a marginally weaker form of Conjecture~\ref{conj:trianglebip}, guaranteeing a bipartite induced subgraph of minimum degree at least $C\log d/\log\log d$.

In a subsequent work, Kelly and Postle~\cite[Conj.~7.2]{KePo18+} proposed a particular local strengthening of fractional colouring, in terms of {\em local demands}, and posed a conjecture in this context which would, if confirmed, establish Conjecture~\ref{conj:redux} as a corollary.

In a subsequent work of a subset of the present authors together with Davies~\cite{DJKP18+occupancy}, a generalisation of Conjecture~\ref{conj:redux} (from triangle-free graphs to graphs of a given local edge density) has been proposed.

The same subset of the present authors with Davies~\cite{DJKP18+local} have found a surprisingly short proof for the fractional colouring analogue of Theorem~\ref{thm:Molloy}, via elementary properties of the {\em hard-core model}. Using this, there is a streamlined, self-contained proof (of about two pages overall) of Theorem~\ref{thm:redux}.

\subsubsection*{Acknowledgements}

We thank Ewan Davies for his insightful remark about regular graphs in Conjectures~\ref{conj:redux} and~\ref{conj:redux,edges}.
We thank Matthew Kwan, Benny Sudakov and Tuan Tran for informing us of their independent, concurrent work and of the paper of Poljak and Tuza.

\bibliographystyle{abbrv}
\bibliography{bipind}

\end{document}